\documentclass[11pt]{amsart}
\usepackage{fullpage,cite,hyperref,subcaption}
\usepackage{amssymb,amsmath,amsthm,mathtools,extpfeil}
\usepackage[all,cmtip]{xy}
\usepackage{tikz,tikz-3dplot}
\usetikzlibrary{calc}
\usepackage{rotating,enumitem}
\newtheorem{thm}[equation]{Theorem}
\newtheorem*{thm*}{Theorem}
\newtheorem{thmA}{Theorem}

\newtheorem{lem}[equation]{Lemma}

\theoremstyle{definition}

\numberwithin{equation}{section}

\DeclareMathOperator{\lk}{lk}

\DeclareMathOperator{\geo}{geo}
\newcommand{\ph}{\varphi}
\newcommand{\epsi}{\varepsilon}

\begin{document}
\title{Appendix: The Gromov--Guth--Whitney embedding theorem}
\author{Fedor Manin and Shmuel Weinberger}
\maketitle

\section{Summary}

By using a different method of embedding manifolds in Euclidean space, the bound
of Theorem A can be improved to achieve one tantalizingly close to Gromov's
linearity conjecture:
\begin{thmA} \label{thm:cob}
  Every closed smooth nullcobordant manifold of complexity $V$ has a filling of
  complexity at most $\ph(V)$, where $\ph(V)=o(V^{1+\epsi})$ for every $\epsi>0$.
\end{thmA}
As with the original Theorem A, this holds for both unoriented and oriented
cobordism.

Recall that the polynomial bound on the complexity of a nullcobordism follows
from a quantitative examination of the method of Thom:
\begin{enumerate}
\item One embeds the manifold $M$ in $S^N$, with some control over the shape of a
  tubular neighborhood.
\item This induces a geometrically controlled map from $S^N$ to the Thom space of
  a Grassmannian; one constructs a controlled extension of this map to $D^{N+1}$.
\item Finally, from a simplicial approximation of this nullhomotopy, one can
  extract a submanifold of $D^{N+1}$ which fills $M$ and whose volume is bounded
  by the number of simplices in the approximation.
\end{enumerate}
Part (2) is the result of the quantitative algebraic topology done to control
Lipschitz constants of nullhomotopies.  Abstracting away the method of embedding,
we extract the following:
\begin{thm*}
  Let $M^n$ be an oriented closed smooth nullcobordant manifold which embeds with
  thickness 1 in a ball in $\mathbb{R}^N$ of radius $R$; that is, there is an
  embedding whose exponential map on the unit ball normal bundle is also an
  embedding.  Then $M$ has a filling of complexity at most $C(n,N)R^{N+1}$.  (For
  unoriented cobordism, $C(n,N)R^N$ is sufficient.)
\end{thm*}
\noindent This is optimal in the sense that the asymptotics of the estimates in
steps (2) and (3) cannot be improved.  Then to prove Theorem \ref{thm:cob}, we
simply need the following estimate, which may also be of independent interest.
\begin{thmA}
  Let $M$ be a closed Riemannian $n$-manifold of complexity $V$.  Then for every
  $N \geq 2n+1$, $M$ has a smooth 1-thick embedding $g:M \to \mathbb{R}^N$ into a
  ball of radius
  $$R=C(n,N)V^{\frac{1}{N-n}}(\log V)^{2n+2}.$$
\end{thmA}
This then implies that for every $N$, $M$ has a filling of complexity at most
$$C(n,N)V^{1+\frac{n+1}{N-n}}(\log V)^{(N+1)(2n+2)},$$
proving Theorem \ref{thm:cob}.

The embedding estimate is in turn derived from a similar estimate of Gromov and
Guth \cite{GrGu} for piecewise linear embeddings of simplicial complexes.  The
combinatorial notion of thickness used in that paper does not immediately
translate into a bound on the thickness of a smoothing.  Rather, in order to
prove our estimate we first prove a version of Gromov and Guth's theorem, largely
using their methods, with a stronger notion of thickness which controls what
happens near every simplex.  We then translate this into the smooth world using
the following result.
\begin{thmA}[Corollary of {\cite[Thm.~3]{BDG}}] \label{thm:tri}
  Every Riemannian $n$-manifold of bounded geometry and volume $V$ is
  $C(n)$-bilipschitz to a simplicial complex with $C(n)V$ vertices with each
  vertex lying in at most $L(n)$ simplices.  In particular, every smooth
  $n$-manifold of complexity $V$ has a triangulation with $C(n)V$ vertices and
  each vertex lying in at most $L(n)$ simplices.
\end{thmA}

\subsection*{The PL picture}
In dimensions $<8$ all PL manifolds are smoothable.  Therefore Theorems
\ref{thm:cob} and \ref{thm:tri} together imply that for $n \leq 6$, every PL
nullcobordant manifold with $V$ vertices and at most $L$ simplices meeting at a
vertex admits a PL filling with $C(n,L)\ph(V)$ vertices and at most $L$ simplices
meeting at a vertex, where $\ph(V)=o(V^{1+\epsi})$ for every $\epsi>0$.  For $n=3$,
this complements the result of Costantino and D.~Thurston \cite{CoTh} which gives
bounded geometry fillings of quadratic volume without imposing restrictions on
the local geometry of $M$.

On the other hand, in high dimensions the PL cobordism problem is still open, and
poses interesting issues since unlike in the smooth category, $B\mathbf{PL}$ is
not an explicit compact classifying space for PL structures.  We hope to return
to this in a future paper.

\subsection*{So, is it linear?}
Gromov's linearity conjecture appears even more interesting now that we know that
it is so close to being true.  On the other hand, at least in the oriented case,
linearity cannot be achieved by Thom's method.  Suppose that one could always
produce ``optimally space-filling'' embeddings $M \hookrightarrow S^N$, that is,
1-thick embeddings in a ball of radius $V^{1/N}$.  Even in this case, an oriented
filling would have volume $C(n,N)V^{1+1/N}$.

Moreover, recent results of Evra and Kaufman \cite{EvKa} on high-dimensional
expanders imply that, at least for simplicial complexes, the Gromov--Guth
embedding bound is near optimal and space-filling embeddings of this type cannot
be found.  While $n$-manifolds are quite far from being $n$-dimensional
expanders, it is possible that a similar or weaker but still nontrivial lower
bound can be found.  This would show that Thom's method is not sufficient for
constructing linear-volume unoriented fillings, either.

On the other hand, at the moment we cannot reject the possibility that it is
possible to find linear fillings for manifolds by some method radically different
from Thom's.  In particular, it is completely unclear how to go about looking for
a counterexample to Gromov's conjecture, although we believe that ideas related
to expanders may play a key role.

\section{PL embeddings with thick links}

In \cite{GrGu}, Gromov and Guth describe ``thick'' embeddings of $k$-dimensional
simplicial complexes in unit $n$-balls, for $n \geq 2k+1$.  They define the
thickness $T$ of an embedding to be the maximum value such that disjoint
simplices are mapped to sets at least distance $T$ from each other.
\cite[Thm.~2.1]{GrGu} gives a nearly sharp upper bound on the optimal thickness
of such an embedding in terms of the volume and bounds on the geometry.

This condition is insufficient to produce smooth embeddings of bounded geometry,
because as thickness decreases, adjacent 1-simplices of length $\sim 1$ may make
sharper and sharper angles.  In this section we show that Gromov and Guth's
construction can be improved to obtain embeddings that also have large angles.
Recall that the \emph{link} $\lk\sigma$ of a $i$-simplex $\sigma$ inside a
simplicial complex $X$ is the simplicial complex obtained by taking the locus of
points at any sufficiently small distance $\epsi>0$ from any point of $\sigma$ in
all directions normal to $\sigma$.  This complex contains an $(r-i)$-simplex for
every $r$-simplex of $X$ incident to $p$.  If $X$ is linearly embedded in
$\mathbb{R}^n$, there is an obvious induced embedding $\lk\sigma \to S^{n-i-1}$.
We show the following:
\begin{thm} \label{thm:PL}
  Suppose that $X$ is a $k$-dimensional simplicical complex with $V$ vertices and
  each vertex lying in at most $L$ simplices.  Suppose that $n \geq 2k+1$.  Then
  there are $C(n,L)$ and $\alpha(n,L)>0$ and a subdivision $X^\prime$ of $X$ which
  embeds linearly into the $n$-dimensional Euclidean ball of radius
  $$R \leq C(n,L)V^{\frac{1}{n-k}}(\log V)^{2k+2}$$
  with Gromov--Guth thickness 1 and such that for any $i$-simplex $\sigma$ of
  $X^\prime$, the induced embedding $\lk\sigma \to S^{n-i-1}$ is
  $\alpha(n,L)$-thick.
\end{thm}
\begin{proof}
  The proof proceeds with the same major steps as in \cite{GrGu}.  We first show
  that a random linear embedding which satisfies the condition that all links are
  thick, while not having the right thickness, is sparse in a weaker sense: most
  balls have few simplices crossing them.  Gromov and Guth then show that the
  simplices can be bent locally, at a smaller scale, in order to thicken the
  embedding; this produces a linear embedding of a finer complex.  We note that
  if the scale is small enough, this finer, bent embedding also has thick links.

  We write $A \lesssim B$ for $A \leq C(n,L)B$ and $A \sim B$ to mean
  $B \lesssim A \lesssim B$.  Following Gromov--Guth, we actually embed $X$ in a
  $V^{\frac{1}{n-k}}$-ball with thickness $\sim (\log V)^{-(2k+2)}$; for simplicity,
  write $R=V^{\frac{1}{n-k}}$.

  We start by choosing, uniformly at random, an assignment of the vertices of $X$
  to points of $\partial B_R$ from those such that for some $\alpha_0(n,L)>0$, the
  following hold:
  \begin{enumerate}[leftmargin=*]
  \item Adjacent vertices are mapped to points at least distance $\alpha_0R$
    apart.
  \item The linear extension to an embedding of $X$ has $\alpha_0$-thick links.
  \end{enumerate}
  We call the resulting linear embedding $I_0(X)$.  We can choose $\alpha_0$ so
  that this is possible since the thickness of the link of some vertex $v$ (and
  of incident higher-dimensional simplices) only depends on the placement of
  vertices at most distance 2 away.  Moreover, this implies the following:
  \begin{enumerate}[leftmargin=*]
  \item[($*$)] The probability distribution of $v$ conditional on some prior
    distribution on the other vertices is pointwise $\lesssim$ the uniform
    distribution.  This follows from the fact that this is true even when all
    vertices within distance 2 from $v$ are fixed.

    This implies that given a $d$-simplex $\sigma$, the probability distribution
    of $\sigma$ (conditional on any distribution on the vertices outside
    $\sigma$) is likewise pointwise $\lesssim$ the uniform distribution where
    every vertex is mapped independently.
  \item[($\dagger$)] If $d(v,w) \leq 2$, then $v$ and $w$ are mapped at least
    $c_0(n,L)R$ units apart.  In particular, every embedded edge has length
    $\sim R$.
  \end{enumerate}
  \begin{lem}
    With high probability, each unit ball $B_1(p) \subset B_R$ meets
    $\lesssim \log V$ simplices of $I_0(X)$.
  \end{lem}
  \begin{proof}
    By an argument of Gromov--Guth, the probability that a random $B_1(p)$ meets
    a fixed $d$-simplex $\sigma$ is $\lesssim V^{-1}$.

    Therefore, the expected number of simplices hitting $B_1(p)$ is $\lesssim 1$.
    If each simplex hitting $B_1(p)$ was an independent event, then the
    probability that $S$ simplices meet $B_1(p)$ would be $\lesssim e^{-S}$;
    therefore, with high probability, for every $p$ the number of simplices
    hitting $B_1(p)$ would be $\lesssim \log V$.  Indeed, complete independence
    is not necessary for this; the condition ($*$) is sufficient.

    This condition holds when the simplices have no common vertices.  Therefore,
    we can finish with a coloring trick, as in Gromov--Guth.  We color the
    simplices of $X$ so that any two simplices that share a vertex are different
    colors.  This can be done with $(k+1)L$ colors.  With high probability, the
    number of simplices of each color meeting $B_1(p)$ is $\lesssim \log V$.
    Since the number of colors is $\lesssim 1$, we are done.
  \end{proof}
  Now we decompose each simplex into finer simplices, using the family of
  edgewise subdivisions due to Edelsbrunner and Grayson \cite{EdGr}.  This is a
  family of subdivisions of the standard $d$-simplex with parameter $L$ which has
  the following relevant properties:
  \begin{itemize}
  \item All links of interior vertices are isometric, and all links of boundary
    vertices are isometric to part of the interior link.
  \item The subdivided simplices fall into at most $\frac{d!}{2}$ isometry
    classes.  In particular, all edges have length $\sim 1/L$.
  \end{itemize}
  When we apply the edgewise subdivision with parameter $L$, with the appropriate
  linear distortion, to $I_0(X)$, we get an embedding $I_0(X')$ of a subdivided
  complex $X'$ such that all edges have length $\sim 1$ by ($\dagger$) and all
  links have thickness $\gtrsim \alpha_0$ and hence $\gtrsim 1$.

  Now we use the following lemma of Gromov--Guth:
  \begin{lem}
    For every $0<\tau \leq 1$, there is a way to move the vertices of $X'$ by
    $\leq \tau$ such that the resulting embedding $I_\tau(X)$ is
    $\gtrsim \tau\cdot(\log V)^{-(2k+2)}$-thick.
  \end{lem}
  If we choose $\tau(n,L)$ sufficiently small compared to the edge lengths of
  $I(X')$, then there is an $\alpha(n,L)$ such that however we move vertices by
  $\leq \tau$, the links will still be $\alpha$-thick.  Since these edge lengths
  are uniformly bounded below, this completes the proof.
\end{proof}

\section{Thick smooth embeddings}

We now use Theorem \ref{thm:PL} to build thick smooth embeddings of manifolds of
bounded geometry.
\begin{thm}
  Let $M$ be a closed Riemannian $m$-manifold with $\geo(M) \leq 1$ and volume
  $V$.  Then for every $n \geq 2m+1$, there is a smooth embedding
  $g:M \to \mathbb{R}^n$ such that
  \begin{itemize}
  \item $g(M)$ is contained in a ball of radius
    $R=C(m,n)V^{\frac{1}{n-m}}(\log V)^{2m+2}$.
  \item For every unit vector $v \in TM$,
    $$K_0(m,n)R \leq \lvert Dg(v) \rvert \leq K_1(m,n)R.$$
  \item The \emph{reach} of $g$ is greater than $1$, that is, the extension of
    $g$ to the exponential map on the normal bundle of vectors of length
    $\leq 1$ is an embedding.
  \end{itemize}
\end{thm}
\begin{proof}
  We prove this by reducing it to Theorem \ref{thm:PL}.  That is, first we build
  a simplicial complex which is bilipschitz to $M$, with a bilipschitz constant
  depending only on $m$.  We apply Theorem \ref{thm:PL} to this complex to obtain
  a PL embedding, and then smooth it out, using the fact that PL embeddings in
  the Whitney range are always smoothable.  The quantitative bound on the
  smoothing follows from the fact that the local behavior of the PL embedding
  comes from a compact parameter space, allowing us to choose from a compact
  parameter space of local smoothings.

  Throughout this proof we write $A \lesssim B$ to mean $A \leq C(m,n)B$.  This
  is different from the usage in the previous section.

  The first step is achieved by the following result.
  \begin{thm}
    There is a simplicial complex $X$ with at most $L=L(m)$ simplices meeting at
    each vertex and a homeomorphism $h:X \to M$ which is $\ell$-bilipschitz for
    some $\ell=\ell(m)$ when $X$ is equipped with the standard simplexwise
    metric.
  \end{thm}
  \begin{proof}
    We start by constructing an $\epsi$-net $x_1,\ldots,x_V$ of points on $M$ for
    an appropriate $\epsi=\epsi(m)>0$.  We do this greedily: once we've chosen
    $x_1,\ldots,x_t$, we choose $x_{t+1}$ so that it is outside
    $\bigcup_{i=1}^t B_\epsi(x_i)$.  In the end we get a set of points such that the
    $\frac{\epsi}{2}$-balls around them are disjoint and the $\epsi$-balls cover
    $M$.

    Now, \cite[Theorem 3]{BDG} in particular gives the following:
    \begin{lem}
      If $\epsi(m)$ is small enough, there is a perturbation of $x_1,\ldots,x_V$
      to $x_1',\ldots,x_V' \in M$ and a simplicial complex $X$ with a bilipschitz
      homeomorphism $X \to M$ as well as the following properties:
      \begin{itemize}
      \item Its vertices are $x_1',\ldots,x_V'$.
      \item It is equipped with the piecewise linear metric determined by edge
        lengths $d(x_i',x_j')$ which are geodesic distances in $M$.
      \item Its simplices have ``thickness'' $\geq C(m)$; this is defined to be
        the ratio of the least altitude of a vertex above the opposite face to
        the longest edge length.  In particular, since the edge lengths are
        $\sim \epsi$, this means that each simplex is $C(m)$-bilipschitz to a
        standard one.
      \end{itemize}
    \end{lem}
    This automatically gives a bilipschitz map to $X$ with the standard
    simplexwise metric.  Moreover, since $M$ has sectional curvatures $\leq 1$,
    we immediately get a uniform bound on the local combinatorics of $X$.
  \end{proof}
  After applying this result to get $h:X \to M$, we apply Theorem \ref{thm:PL},
  finding an embedding $X' \to \mathbb{R}^n$ of a subdivision $X'$ of $X$ which
  is 1-thick, lands in an $R$-ball for $R=C(m)V^{\frac{1}{n-k}}(\log V)^{2k+2}$,
  has $\alpha(m)$-thick links, and expands all intrinsic distances by
  $\sim R$.  In other words, we get a PL embedding $f:M \to \mathbb{R}^n$.

  For the sake of uniformity, we expand the metric of $M$ by a factor of $R$;
  then we see that the embedding $f$ is locally uniformly bilipschitz.  That is,
  for any $x,y \in M$ such that $d(x,y) \leq 1$,
  $$d(f(x),f(y)) \sim d(x,y).$$

  As in the main part of the paper, we assume that $M$ additionally has
  controlled $k$th covariant derivatives of its curvature tensor for every $k$.
  This allows us, like in Lemma 6.1, to fix an atlas
  $\mathfrak{U}=\{\phi_i:B_\mu \to M\}$ for $M$, with the following properties:
  \begin{enumerate}
  \item The $\phi_i(B_{\mu/2})$ also cover $M$.
  \item $\mathfrak{U}$ is the disjoint union of sets
    $\mathfrak{U}_1,\ldots,\mathfrak{U}_r$ each consisting of pairwise disjoint
    charts.
  \item The charts are uniformly bilipschitz, and the $k$th derivatives of all
    transition maps between charts are uniformly bounded depending only on $m$
    and $n$.
  \end{enumerate}
  Here $\mu$ and $r$ both depend only on $m$ and $n$.  We construct our smoothing
  first on $\mathfrak{U}_1$, then extend to $\mathfrak{U}_2$, and so on by
  induction.

  At each step of the induction, we use the following form of the weak Whitney
  embedding theorem \cite[\S2.2, Thm.~2.13]{Hirsch}: for $s \geq 2r+1$, the set
  of smooth embeddings $D^r \to \mathbb{R}^s$ is $C^0$-dense in the set of
  continuous maps.  Moreover, the set of smooth maps which restrict to some
  specific smooth map on a closed codimension zero submanifold is likewise dense
  in the set of such continuous maps \cite[\S2.2, Ex.~4]{Hirsch}.

  The strategy is as follows.  Note that the space of $L$-bilipschitz maps
  $B_\mu \to \mathbb{R}^n$ up to translation is compact by the Arzel\`a--Ascoli
  theorem.  At every stage, we also have a $C^\infty$-compact space of possible
  partial local smoothings.  Then Whitney will allow us to choose an extension
  from a space of possibilities which is also $C^\infty$-compact.

  We now give a detailed account of the inductive step.  Suppose that we have
  defined a partial smooth embedding $g:K \to \mathbb{R}^n$, where $K$ is a
  compact codimension zero submanifold of $M$ with
  \begin{equation} \label{domain}
    \bigcup_{\substack{\phi \in \mathfrak{U}_i\\1 \leq i<j}}
    \phi\left(B_{\mu\cdot\frac{2r-j}{2r}}\right) \subset K \subset
    \bigcup_{\substack{\phi \in \mathfrak{U}_i\\1 \leq i<j}} \phi(B_\mu).
  \end{equation}
  Moreover, suppose that $g$ is $\rho_{j-1}$-close to $f$ for some sufficiently
  small $\rho_{j-1}$ depending on $m$ and $n$, and that for each
  $\phi \in \mathfrak{U}_i$, $i<j$, the partially defined function $g \circ \phi$
  is an element of a $C^\infty$-compact moduli space $\mathcal{L}_{j-1}$ of maps
  each from one of a finite set of subdomains of $B_\mu$ to $D^n$.

  Fix a fine cubical mesh in $B_\mu$; it should be fine enough that any transition
  function sends a distance of $\mu/2r$ to at least four times the diagonal of
  the cubes.  The purpose of this mesh is to provide a uniformly finite set of
  subsets on which maps may be defined.  Then, again by Arzel\`a--Ascoli, for any
  set $K$ which is a union of cubes in this mesh, the space of potential
  transition maps $K \to B_\mu$ satisfying the bounds on the covariant derivatives
  in all degrees is $C^\infty$-compact.

  Fix $\phi \in \mathfrak{U}_j$.  By the above, $g|_{K \cap \phi(B_\mu)} \circ \phi$,
  again restricted to the union of cubes on which it is fully defined (call this
  domain $\hat K \subset B_\mu$), is also chosen from a $C^\infty$-compact moduli
  space $\mathcal{M}_j$, whose elements are patched together from a bounded
  number of compositions of elements of $\mathcal{L}_{j-1}$ with transition maps
  as above.  Of course, $\mathcal{M}_1$ consists of the unique map from the empty
  set.

  Let $\mathcal{N}_j$ be the $C^0$-compact set of $L$-bilipschitz embeddings
  $B_{\mu(1-1/2r)} \to D^n$.  Notice that the subset
  $\Delta \subset \mathcal{M}_j \times \mathcal{N}_j$ consisting of pairs whose
  $C^0$ distance is $\leq \rho_{j-1}$ is compact; this $\Delta$ contains the pair
  $(g|_{\hat K} \circ \phi,f \circ \phi)$.

  Fix a smooth embedding $u:B_{\mu(1-1/2r)} \to D^n$.  We say that
  $(\ph,\psi) \in \Delta$ is \emph{$\epsi$-good for} $u$, for some $\epsi>0$, if:
  \begin{itemize}
  \item The $C^0$ distance between $u$ and $\psi$ is $<\rho_j$, where
    $\rho_j>\rho_{j-1}$ is fixed.
  \item The map interpolating between $\phi$ and $u$ via a bump function, only
    depending on $\hat K$, whose transition lies within the layer of cubes
    touching the boundary of $\hat K$, has reach $>\epsi$.  (Here, we simply
    delete all boundary cubes outside of $B_{\mu\cdot(1-1/2r)}$ from the domain.
    Thus at this step the domain of $g$ actually recedes slightly; this is the
    motivation for the condition \eqref{domain}.)
  \end{itemize}
  For any fixed pair $(u,\epsi)$, these are both open conditions in $\Delta$, so
  there is an open set $V_{u,\epsi} \subseteq \Delta$ of good pairs $(\ph,\psi)$.
  Moreover, since (by Whitney) we can always choose a $u$ which coincides on
  $\hat K$ with a given element of $\mathcal{M}_j$, these sets cover $\Delta$.
  Therefore we can take a finite subcover corresponding to a set of pairs
  $(u_i,\epsi_i)$.  Taking a cover by compact subsets subordinate to this, we get
  a compact set of allowable extensions of elements of $\mathcal{M}_j$ to
  $B_{\mu\cdot(1-1/2r)}$; together with the modified sets of allowable maps on
  previous $\mathfrak{U}_i$'s (cut back so as to be defined on a domain of cubes)
  this makes $\mathcal{L}_j$.

  We choose an extension of $g$ from the set of allowable extensions above.
  Doing this for every $\phi \in \mathfrak{U}_j$ completes the induction step,
  giving some bound on the local geometry and reach by the compactness argument.
  Moreover, if we pick $\rho_j$ small enough compared to $\mu/2r$, then the
  embedding outside $\phi(B_\mu)$ stays far enough away from the embedding inside.
  Nevertheless, all of these bounds become worse with every stage of the
  induction.

  At the end of the induction, we have a smooth embedding of $M$.  Every choice
  we made was from a compact set of local smoothings depending ultimately only on
  $m$ and $n$, which in turn controlled various bilipschitz and $C^k$ bounds.
  Thus the resulting submanifold $\tilde M=g(M) \subset B_R$ has
  $\geo(\tilde M) \lesssim 1$.  For the same reason, $g$ (as a map from $M$ with
  its original metric) has all directional derivatives $\sim R$.  Moreover, since
  we didn't move very far from $f$, points from disjoint simplices can't have
  gotten too close to each other.  This, together with the local conditions,
  shows that $\tilde M$ has an embedded normal bundle of radius $\gtrsim 1$.  By
  expanding everything by some additional $C(m,n)$ we achieve the bounds desired
  in the statement of the theorem.
\end{proof}

\subsection*{Acknowledgements}

The authors would like to thank Larry Guth and Sasha Berdnikov for stimulating
correspondence during the course of this work, and two anonymous referees for
helpful comments regarding the exposition.

\bibliographystyle{amsalpha}
\bibliography{eff_whitney}
\end{document}